\documentclass[11pt,a4paper]{amsart}[2000]

\usepackage{amsmath}
\usepackage{amssymb}
\usepackage{amsthm}
\usepackage[hang]{footmisc}
\usepackage[all]{xypic}
\usepackage{hyperref}
\usepackage{color}
\usepackage{mathtools}
\usepackage{enumitem}
\usepackage{stackengine}

\title[A categorical study on the generalized type semigroup]{A categorical study on the generalized type semigroup}

\thanks{}

\allowdisplaybreaks

\theoremstyle{plain}

\newtheorem{Thm}{Theorem}[section]

\theoremstyle{definition}

\theoremstyle{plain}
\newtheorem{thm}[Thm]{Theorem}
\newtheorem{lem}[Thm]{Lemma}
\newtheorem{cor}[Thm]{Corollary}
\newtheorem{prop}[Thm]{Proposition}

\theoremstyle{definition}
\newtheorem{defn}[Thm]{Definition}

\newtheorem{rmk}[Thm]{Remark}

\newcommand{\B}{B}
\newcommand{\A}{A}
\newcommand{\J}{J}
\newcommand{\K}{\mathcal{K}}
\newcommand{\D}{D}
\newcommand{\Ch}{D}
\newcommand{\Zh}{\mathcal{Z}}
\newcommand{\E}{E}
\newcommand{\Oh}{\mathcal{O}}

\newcommand{\T}{{\mathbb T}}

\newcommand{\N}{{\mathbb N}}
\newcommand{\Z}{{\mathbb Z}}
\newcommand{\C}{{\mathbb C}}
\newcommand{\Q}{{\mathbb Q}}

\newcommand{\aut}{\mathrm{Aut}}
\newcommand{\supp}{\mathrm{supp}}

\newcommand{\eps}{\varepsilon}
\numberwithin{equation}{section}

\newcommand{\id}{\mathrm{id}}

\newcommand{\halpha}{\widehat{\alpha}}
\newcommand{\calpha}{\widehat{\alpha}}

\newcommand{\tih}{\widetilde {h}}

\newcommand\set[1]{\left\{#1\right\}}  
\newcommand\mset[1]{\left\{\!\!\left\{#1\right\}\!\!\right\}}

\newcommand{\CA}[0]{\mathcal{A}} \newcommand{\CB}[0]{\mathcal{B}}
\newcommand{\CC}[0]{\mathcal{C}} \newcommand{\CD}[0]{\mathcal{D}}
 
\newcommand{\CG}[0]{\mathcal{G}} \newcommand{\CH}[0]{\mathcal{H}}
 
\newcommand{\CK}[0]{\mathcal{K}} 
 
\newcommand{\CO}[0]{\mathcal{O}} 
\newcommand{\CQ}[0]{\mathcal{Q}} 
 \newcommand{\CT}[0]{\mathcal{T}}
 
\newcommand{\CW}[0]{\mathcal{W}}

\newcommand{\Ra}[0]{\Rightarrow}
\newcommand{\La}[0]{\Leftarrow}
\newcommand{\LRa}[0]{\Leftrightarrow}

\newcommand{\quer}[0]{\overline}
\newcommand{\eins}[0]{\mathbf{1}}			
\newcommand{\diag}[0]{\operatorname{diag}}
\newcommand{\ad}[0]{\operatorname{Ad}}
\newcommand{\ev}[0]{\operatorname{ev}}
\newcommand{\fin}[0]{{\subset\!\!\!\subset}}
\newcommand{\diam}[0]{\operatorname{diam}}
\newcommand{\Hom}[0]{\operatorname{Hom}}
\newcommand{\dst}[0]{\displaystyle}
\newcommand{\spp}[0]{\operatorname{supp}}
\newcommand{\lsc}[0]{\operatorname{Lsc}}
\newcommand{\del}[0]{\partial}
\newcommand{\GU}[0]{\CG^{(0)}}

\newcommand{\cu}[0]{\operatorname{Cu}}

\theoremstyle{definition}

\numberwithin{equation}{Thm}

\begin{document}
\global\long\def\floorstar#1{\lfloor#1\rfloor}
\global\long\def\ceilstar#1{\lceil#1\rceil}	

\global\long\def\B{B}
\global\long\def\A{A}
\global\long\def\J{J}
\global\long\def\K{\mathcal{K}}
\global\long\def\D{D}
\global\long\def\Ch{D}
\global\long\def\Zh{\mathcal{Z}}
\global\long\def\E{E}
\global\long\def\Oh{\mathcal{O}}

\global\long\def\T{{\mathbb{T}}}
\global\long\def\BR{{\mathbb{R}}}
\global\long\def\N{{\mathbb{N}}}
\global\long\def\Z{{\mathbb{Z}}}
\global\long\def\C{{\mathbb{C}}}
\global\long\def\Q{{\mathbb{Q}}}

\global\long\def\aut{\mathrm{Aut}}
\global\long\def\supp{\mathrm{supp}}

\global\long\def\eps{\varepsilon}

\global\long\def\id{\mathrm{id}}

\global\long\def\halpha{\widehat{\alpha}}
\global\long\def\calpha{\widehat{\alpha}}

\global\long\def\tih{\widetilde{h}}

\global\long\def\opFol{\operatorname{F{\o}l}}

\global\long\def\opRange{\operatorname{Range}}

\global\long\def\opIso{\operatorname{Iso}}

\global\long\def\dimnuc{\dim_{\operatorname{nuc}}}

\global\long\def\set#1{\left\{  #1\right\}  }


\global\long\def\mset#1{\left\{  \!\!\left\{  #1\right\}  \!\!\right\}  }

\global\long\def\Ra{\Rightarrow}
\global\long\def\La{\Leftarrow}
\global\long\def\LRa{\Leftrightarrow}

\global\long\def\quer{\overline{}}
\global\long\def\eins{\mathbf{1}}
\global\long\def\diag{\operatorname{diag}}
\global\long\def\ad{\operatorname{Ad}}
\global\long\def\ev{\operatorname{ev}}
\global\long\def\fin{{\subset\!\!\!\subset}}
\global\long\def\diam{\operatorname{diam}}
\global\long\def\Hom{\operatorname{Hom}}
\global\long\def\dst{{\displaystyle }}
\global\long\def\spp{\operatorname{supp}}
\global\long\def\spo{\operatorname{supp}_{o}}
\global\long\def\del{\partial}
\global\long\def\lsc{\operatorname{Lsc}}
\global\long\def\GU{\CG^{(0)}}
\global\long\def\HU{\CH^{(0)}}
\global\long\def\AU{\CA^{(0)}}
\global\long\def\BU{\CB^{(0)}}
\global\long\def\CUU{\CC^{(0)}}
\global\long\def\DU{\CD^{(0)}}
\global\long\def\QU{\CQ^{(0)}}
\global\long\def\TU{\CT^{(0)}}
\global\long\def\CUUU{\CC'{}^{(0)}}

\global\long\def\AUl{(\CA^{l})^{(0)}}
\global\long\def\BUl{(B^{l})^{(0)}}
\global\long\def\HUp{(\CH^{p})^{(0)}}

\global\long\def\properlength{proper}

\global\long\def\interior#1{#1^{\operatorname{o}}}
	
\author{Xin Ma}
\email{xma1@memphis.edu}
\address{Department of Mathematics,
	University of Memphis,
	Memphis, TN, 38152}


\date{June 19, 2022.}

	
\begin{abstract}
In this short note, we show that the generalized type semigroup $\CW(X, \Gamma)$ introduced by the author in \cite{M3} belongs to the category \textnormal{W} introduced in \cite{APT}. In particular, we demonstrate that $\CW(X, \Gamma)$  satisfies axioms (W1)-(W4) and (W6). When $X$ is zero-dimensional, we also establish (W5) for the semigroup. This supports the analogy between the generalized type semigroup and  pre-completed Cuntz semigroup $W(\cdot)$ for $C^*$-algebras.
\end{abstract}
\maketitle

\section{Introduction}
The Cuntz semigroup $W(A)$ and its complete version $\operatorname{Cu}(A)$ of a $C^*$-algebra $A$ have become important invariants in the structure theory of $C^*$-algebras.   A categorical study of (completed) Cuntz semigroup $\operatorname{Cu}(\cdot)$ was initiated in \cite{CEI} by introducing the category Cu for semigroups. Then the pre-completed version parallel to the category Cu, called the category W, corresponding to the pre-completed Cuntz semigroup $W(\cdot)$ for $C^*$-algebras, was introduced in \cite{APT}. The properties of the categories Cu and W as well as their relation with the category of $C^*$-algebras was further developed  in, e.g., \cite{APT}, \cite{APT1} and \cite{APT2}.  In particular, like in the original $C^*$-setting that $\operatorname{Cu}(A)=W(A\otimes \CK)$, as a ``completion'' of $W(A)$, it was shown in \cite{APT} that every W-semigroup $S$ can be completed to a $\cu$-semigroup $\gamma(S)$ in the sense of \cite[Definition 3.1.7]{APT} by using a standard process described in \cite[Proposition 3.1.6]{APT}. Here,  $\gamma$ is the completion functor  witnessed that category $\operatorname{Cu}$ is a full reflexive subcategory of the category W. See  \cite[Section 3]{APT} for more details. 

In \cite{CEI}, four structural axioms (O1)-(O4) have been introduced for a positively partially  ordered monoid $S$ to become a $\operatorname{Cu}$-semigroup. Additional 
axioms (O5) and (O6) for a $\operatorname{Cu}$-semigroup $S$ were introduced in \cite{RW} and \cite{R}, and also considered in \cite{APT}. Roughly speaking, axiom (O5) measures how far the order in $S$ is from being the algebraic order, whilst axiom (O6) measures the Riesz decomposition present in $S$. We remark that the Cuntz semigroup $\operatorname{Cu}(A)$ satisfies (O1)-(O6) for any $C^*$-algebra $A$. See, e.g., \cite{CEI} and \cite[Proposition 4.6]{APT}. The category W was defined in \cite{APT} by means of four axioms (W1)-(W4), parallel to (O1)-(O4). Additional axioms (W5) and (W6) were also introduced, to play the role of (O5) and (O6) for these uncompleted semigroups; see Section 2 and 3 below.

Then a natural question is that, besides $C^*$-algebras, are there any other sources that $\cu$- and W-semigroups may come from? An obvious candidate is the class of the \textit{generalized type semigroups} introduced by the author in \cite{M3} as a dynamical analogue of Cuntz semigroups of $C^*$-algebras. In general, it seems that a generalized type semigroup does not come as a Cuntz semigroup of a $C^*$-algebra.

Suppose that $\alpha: \Gamma\curvearrowright X$ is a continuous action of a countably infinite discrete group $\Gamma$ on a locally compact Hausdorff second countable space $X$. The generalized type semigroup $\CW(X, \Gamma)$ was originally constructed to study the almost unperforation form of the dynamical comparison, as a dynamical analogue of strict comparison in the $C^*$-setting. Dynamical comparison was first introduced by Winter in 2012 and refined by Kerr in \cite{D} to study regularity properties of the reduced crossed product $C^*$-algebra $C_0(X)\rtimes_r \Gamma$. In this short note, we further investigate the generalized type semigroup $\CW(X, \Gamma)$ from a categorical viewpoint. In particular, we show that $\CW(X, \Gamma)$ belongs to the category W by showing that it satisfies the axioms (W1)-(W4) and thus its completion $\gamma(\CW(X, \Gamma))$ belongs to the category Cu. We also establish (W6) for $\CW(X, \Gamma)$ and thus  the completion $\gamma(\CW(X, \Gamma))$ satisfies (O6) by \cite[Theorem 4.4]{APT}. In addition, when $X$ is zero-dimensional, we establish (W5) for $\CW(X, \Gamma)$ and thus  the completion $\gamma(\CW(X, \Gamma))$ satisfies (O5) by \cite[Theorem 4.4]{APT} as well. We begin with recalling the definition of $\CW(X, \Gamma)$. We remark that, unlike the original definition in \cite{M3}, we allow our space $X$ to be locally compact instead of only being compact. Such a version has appeared in \cite{M4} written in the language of \'{e}tale groupoid. It appears very plausible that the results in this note can be generalized to the setting of locally compact Hausdorff \'{e}tale groupoids. In this note, we use notations ``$A\sqcup B$'' and ``$\bigsqcup\CA$'' for disjoint unions.

\begin{defn}\label{defn: subequivalence relation}
Let $O_1,\dots, O_n$ and $V_1,\dots, V_m$ be two sequences of open sets in $X$, We write \[\bigsqcup_{i=1}^n O_i\times \{i\}\preccurlyeq \bigsqcup_{l=1}^m V_l\times \{l\}\]
	if for every $i\in\{1,\dots, n\}$ and every compact set $F_i\subset O_i$ there are a collection of open sets, $\mathcal{U}_i=\{U^{(i)}_1,\dots, U^{(i)}_{J_i}\}$ forming a cover of $F_i$, $s^{(i)}_1,\dots, s^{(i)}_{J_i}\in \Gamma$ and  $k_1^{(i)},\dots, k_{J_i}^{(i)}\in \{1,\dots, m\}$ such that
	\[\bigsqcup_{i=1}^n\bigsqcup_{j=1}^{J_i}s^{(i)}_jU^{(i)}_j\times \{k^{(i)}_j\}\subset \bigsqcup_{l=1}^m V_l\times \{l\}.\]
We allow the empty set $\emptyset$ to appear as one or more of the open sets in Definition \ref{defn: subequivalence relation} with no harm. In fact we make sense of $\emptyset\preccurlyeq \bigsqcup_{l=1}^m V_l\times \{l\}$ for any $m\in \N^+$ and open sets $V_l$. Note that the empty set $\emptyset$ above could be interpreted as $\bigsqcup_{i=1}^n \emptyset\times \{i\}$ for any $n\in \N^+$. We also emphasize that each $V_l$ above could also be the empty set. 
\end{defn}

Denote by $\CO_n(X)$ for the set $\{(O_1,\dots, O_n): O_i\ \text{is open in }X\ \text{for any } i\leq n\}$. Write $\CK(X, \Gamma)=\bigcup_{n=1}^\infty \CO_n(X)$. We also denote by $\CO(X)=\CO_1(X)$ the collection of all open sets in $X$ for simplicity. Note that the relation in Definition \ref{defn: subequivalence relation} is defined on $\CK(X, \Gamma)$. The following simple but fundamental result was established in \cite{M3}.

\begin{lem}\cite[Lemma 2.2]{M3}\label{lem:transitivity}
The relation ``$\preccurlyeq$'' on $\CK(X, \Gamma)$ is transitive.
\end{lem}

Then Lemma \ref{lem:transitivity} allows one to define  an equivalence relation on $K(X, \Gamma)$ by setting $a\approx b$ if $a\preccurlyeq b$ and $b\preccurlyeq a$ for $a,b\in \CK(X, \Gamma)$. Thus we have the following definition.

\begin{defn}(\cite{M3})
Write  $\CW(X, \Gamma)$ for the quotient $\CK(X, \Gamma)/\approx$ and define an operation on $\CW(X, \Gamma)$ by $[a]+[b]=[(a,b)]$, where $(a,b)$ is defined to be the concatenation of $a=(O_1,\dots, O_n)$ and $b=(U_1,\dots, U_m)$, i.e., $(a,b)=(O_1,\dots, O_n,U_1,\dots, U_m)$. It can be additionally verified that this operation is well defined and commutative, i.e, $[a]+[b]=[b]+[a]$.  Moreover, we endow $\CW(X, \Gamma)$ with the natural order by declaring $[a]\leq [b]$ if $a\preccurlyeq b$. Thus $\CW(X, \Gamma)$ is a well-defined commutative partially ordered semigroup.
\end{defn}

Let $S$ be a monoid and $\leq$  a preorder on $S$. We say $(S, \leq)$ is a \textit{preordered monoid} if $\leq$ is translation-invariant, i.e., $a\leq b$ implies that $a+c\leq b+c$ for any $a,b,c\in S$. If, in addition, $0\leq a$ holds for any $a\in S$, we say $(S, \leq)$ is a \textit{positively ordered monoid}.  Moreover, if $\leq$ is a partial order, we say $(S, \leq)$ is a \textit{positively partially ordered monoid}. For $a\in S$ and $<$ a transitive relation on $S$, we denote by $a^<$ the set $\{b\in S:b<a\}$.

\begin{rmk}\label{rmk: monoid structure on K(X, Gamma)}
 It is direct to see that our $\CW(X, \Gamma)$ is a positively partially ordered monoid. In addition, if we slightly abuse the notation as we mentioned in Definition \ref{defn: subequivalence relation} for the empty set $\emptyset$, the pair $(\CK(X, \Gamma), \preccurlyeq)$ itself is a positively ordered monoid with respect to the addition defined as $a+b=(a, b)$.
\end{rmk}

\section{The category \textnormal{W}}
In this section, we establish that each $\CW(X, \Gamma)$ belongs to the category \textnormal{W} by proving that it satisfies the axioms (W1)-(W4). We begin with introducing a new  relation on $\CK(X, \Gamma)$. In general, let $(S, \leq)$ be a positively ordered monoid and $a, b\in S$. We say $a$ is \textit{way-below} $b$, denoted by $a\ll b$, if whenever $\{b_n: n\in \N\}$ is an increasing sequence in $S$ such that the supremum $\sup_nb_n$ exists, then $b\leq \sup_nb_n$ implies that there is an $m\in \N$ such that $a\leq b_m$. Equipped $\CO(X)$ with the usual inclusion relation,  it is not hard to see $O\ll U$ in $(\CO(X), \subset)$ is equivalent to that $O\subset K\subset U$ for some compact set $K$. One can extend way-below relation to $\CO_n(X)$ as follows. We say that a set $A$ in $X$ is precompact if its closure $\overline{A}$ is compact. Let $a=(O_1, \dots, O_n), b=(U_1,\dots, U_n)\in \CO_n(X)$. Then $a$ is way-below $b$ in $\CO_n(X)$, denoted by $a\ll_n b$, if and only if each $O_i$ is precompact and $\overline{O_i}\subset U_i$ for any $1\leq i\leq n$. The following was established in \cite{M4}. See also \cite[Proposition 3.2]{M3}.

\begin{prop}\label{prop: relation}\cite[Proposition 4.6]{M4}
	Let $a\in \CO_n(X)$ and $b\in \CO_m(X)$. Then the following are equivalent.
	\begin{enumerate}[label=(\roman*)]
		\item $a\preccurlyeq b$;
		
		\item $c\preccurlyeq b$ for any $c\in \CO_n(X)$ with $c\ll_n a$;
		
		\item for any $c\in \CO_n(X)$ with $c\ll_n a$ there is a $d\in\CO_m(X)$ with $d\ll_m b$ such that $c\preccurlyeq d$.
	\end{enumerate}
\end{prop}
We  recall the following definition of an auxiliary relation as it appeared in Paragraph 2.1.1 in \cite{APT} for positively ordered monoids introduced in Section B.2 in \cite{APT}. We also remark that the notion of auxiliary relation originated in the work of continuous lattices and domains. See \cite{G}.
\begin{defn}\cite{APT}
Let $(S, \leq)$ be a positively ordered monoid. An \textit{auxiliary relation} on $(S, \leq)$ is a binary relation $\prec$ such that the following hold:
\begin{enumerate}[label=(\roman*)]
	\item $a\prec b$ implies $a\leq b$ for any $a, b\in S$.
	\item $a\leq b\prec c\leq d$ implies $a\prec d$ for any $a, b, c, d\in S$.
	\item $0\prec a$ for any $a\in S$.
\end{enumerate}
\end{defn}
An element $x\in S$ is said to be \textit{compact} if $x\prec x$.  We next recall axioms (W1)-(W4) introduced in Paragraph 2.1.1 in \cite{APT} for a positively partially ordered monoid $S$. But one may still consider these axioms in the context of preordered monoids.
\begin{enumerate}
	\item[(W1)] For any $a\in S$, there is a sequence $(a_k)_k$ in $a^\prec$ satisfying $a_k\prec a_{k+1}$ for each $k$ and such that for any $b\in a^\prec$ there is a $k$ with $b\prec a_k$.
	\item[(W2)] For any $a\in S$, one has $a=\sup a^\prec$.
	\item[(W3)] If $a', a, b', b\in S$ satisfy $a'\prec a$ and $b'\prec b$, then $a'+b'\prec a+b$. 
	\item[(W4)] If $a, b, c\in S$ satisfy $a\prec b+c$, then there exist $b', c'\in S$ such that $a\prec b'+c'$, $b'\prec b$ and $c'\prec c$.
\end{enumerate}
The category W consists of all positively partially ordered monoids $S$ equipped with an auxiliary relation $\prec$ satisfying (W1)-(W4). We refer such a semigroup as a W-semigroup.

\begin{lem}\label{lem: uniform}
	Let $(S, \preccurlyeq)$ be a positively ordered monoid equipped with another transitive relation $\ll$ satisfying
	\begin{enumerate}
		\item $0\ll a$ for any $a\in S$, and
		\item $\preccurlyeq$ is weaker than $\ll$.
	\end{enumerate}
 Let $\prec$ be a new relation on $S$ defining by $a\prec b$ in $S$ if there is an element $c\in S$ such that $a\preccurlyeq c\ll b$. Suppose $S$ satisfies  that $a\preccurlyeq b$ if and only if $a^{\ll}\subset b^{\prec}$. Then one has 
	\begin{enumerate}[label=(\roman*)]
		\item $\prec$ is auxiliary for $\preccurlyeq$.
		\item If $(S, \ll)$ satisfies \textnormal{(W1), (W3)} and \textnormal{(W4)}, then $(S, \prec)$ satisfies \textnormal{(W1)-(W4).}
	\end{enumerate}
\end{lem}
\begin{proof}
We first establish (i). Let $a\prec b$ in $S$. Then, by definition, there is an element $c\in S$ such that $a\preccurlyeq c\ll b$. Now since the relation $\preccurlyeq$ is weaker than $\ll$, one actually has $a\preccurlyeq c\preccurlyeq b$ and thus $a\preccurlyeq b$. Then, let $a\preccurlyeq b\prec c\preccurlyeq d$ in $S$. By definition, there is an $e\in S$ such that $a\preccurlyeq b\preccurlyeq e\ll c\preccurlyeq d$. Now,  $c\preccurlyeq d$ implies that $c^{\ll}\subset d^\prec$ and thus $e\in d^\prec$. These entail that there is an $f\in S$ such that $a\preccurlyeq e\preccurlyeq f\ll d$ and therefore one has $a\prec d$ by definition. Finally, for any $a\in S$, one has that $0\preccurlyeq 0\ll a$, which means $0\prec a$. These have shown $\prec$ is an auxiliary relation for $(S, \preccurlyeq)$.

Now for (ii),  suppose $(S, \ll)$ satisfies (W1), (W3) and (W4). Since $(S, \ll)$ satisfies (W1), for any $a\in S$, there is a $\ll$-increasing sequence $\{a_k: k\in \N\}$ in $a^\ll$ such that for any $b\in a^\ll$ there is a $k\in \N$ such that $b\ll a_k$. Note that $\ll$ is stronger than $\prec$ by definition and thus  $\{a_k: k\in \N\}$ is actually a $\prec$-increasing sequence in $a^\prec$. Suppose $b\in a^\prec$. then there is an element $c\in S$ such that $b\preccurlyeq c\ll a$, which implies that there is a $k\in \N$ such that $c\ll a_k$. This means $b\prec a_k$ and thus (W1) holds for $(S, \prec)$.

Note that (W2) for $(S, \prec)$ follows from the comments in \cite[Remark 2.6]{APT2}. But, to make the proof self-contained, we establish (W2) for $(S, \prec)$ here as follows.
We first claim that $a\preccurlyeq b$ if and only if $a^\prec\subset b^\prec$. Indeed, if $a\preccurlyeq b$, then for any $c\prec a$, one has $c\prec b$ as well because we have shown above that $\prec$ is auxiliary for $\preccurlyeq$. This implies that $a^\prec\subset b^\prec$. For the other direction, note that $a^\prec\subset b^\prec$ implies $a^\ll\subset b^\prec$, which is equivalent to $a\preccurlyeq b$ by our hypothesis. Now, for any $a\in S$, it is easy to see $a$ is an upper bound for $a^\prec$. On the other hand, (W1) for $\prec$ above implies that $a^\prec=\bigcup_{k=1}^\infty a_k^\prec$ for a suitable sequence $\{a_k: k\in\N\}\subset a^\prec$. Now let $b$ be another upper bound for $a^\prec$, which implies that $a_k^\prec\subset b^\prec$ for any $k\in \N$. Therefore, $a^\prec\subset b^\prec$ holds, which means $a\preccurlyeq b$ and thus (W2) holds for $\prec$.

For (W3), let $a', a, b', b\in S$ such that $a'\prec a$ and $b'\prec b$. Then there are $c, d\in S$ such that $a'\preccurlyeq c\ll a$ and $b'\preccurlyeq d\ll b$. Now since $\ll$ satisfies (W3), one has $a'+b'\preccurlyeq c+d\ll a+b$, which implies $a'+b'\prec a+b$. This means (W3) holds for $\prec$.

Finally, for (W4), let $a\prec b+c$ in $S$. Then there is an element $d\in S$ such that $a\preccurlyeq d\ll b+c$. Now because (W4) holds for $\ll$, there are $b', c'\in S$ satisfying  $b'\ll b$, $c'\ll c$ and $d\ll b'+c'$. Note that these show that $a\prec b'+c'$, $b'\prec b$ and $c'\prec c$ as desired.
\end{proof}

Define a relation $\ll$ on $(\CK(X, \Gamma), \preccurlyeq)$ by setting $a\ll b$ if $a, b\in \CO_n(X)$ for some $n\in \N$ and $a\ll_n b$.  It is direct to see $\ll$ satisfies conditions (1) and (2) in Lemma \ref{lem: uniform} with respect to $\preccurlyeq$ as in Proposition \ref{prop: relation}. Then define another relation $\prec$ on $\CK(X, \Gamma)$ by setting  $a\prec b$ if there is an element $c\in \CK(X, \Gamma)$ such that $a\preccurlyeq c\ll b$. 

\begin{thm}\label{thm: K is W semigroup}
	 $\CK(X, \Gamma)$ satisfies \textnormal{(W1)-(W4)} with respect to the $\prec$.
\end{thm}
\begin{proof}
We first verify that $a\preccurlyeq b$ is equivalent to $a^\ll\subset b^\prec$ in $\CK(X, \Gamma)$. Let $a\in \CO_n(X)$ and $b\in \CO_m(X)$ be $a\preccurlyeq b$.  Let $c\ll_n a$ and then Proposition \ref{prop: relation} implies that there is an element $d\ll_m b$ such that $c\preccurlyeq d$, which means that $c\prec b$. This shows that $a^\ll\subset b^\prec$. For the converse, suppose $a^\ll\subset b^\prec$ holds. Now let $c\ll_n a$. One then has $c\prec b$, which means that there is an element $d\in \CO_m(X)$ such that $c\preccurlyeq d\ll_m b$. Therefore, one has $c\preccurlyeq b$ and Proposition \ref{prop: relation} shows that $a\preccurlyeq b$. Now, Lemma \ref{lem: uniform}(i) entails that $\prec$ is auxiliary for $\preccurlyeq$ in $\CK(X, \Gamma)$.

We then verify $(\CK(X, \Gamma), \ll)$ satisfies (W1), (W3) and (W4). Suppose $a=(O_1,\dots, O_n)\in \CO_n(X)$. Using the facts that each $O_i$ is a $K_\sigma$ set and  the space $X$ is normal, there is an increasing sequence $\{U^i_k:  k\in \N\}$ of open sets such that  each $U^i_k$ is precompact satisfying $\overline{U^i_k}\subset U^i_{k+1}$ and $O_i=\bigcup_{k=1}^\infty U^i_k$ for each $i\leq n$. Then define $a_k=(U^1_k,\dots, U^n_k)$ for $k\in \N$. By our construction, $a_k\ll a_{k+1}$ holds and observe that each $a_k\in a^{\ll}$. Now, let $b=(V_1, \dots, V_n)\ll a$, which means $V_i$ is precompact and $\overline{V_i}\subset O_i$ for each $i\leq n$. Now since $O_i=\bigcup_{k=1}^\infty U^i_k$, there is a big enough $k\in \N$ such that $\overline{V_i}\subset U^i_k$ for each $i\leq n$. This implies that $b\ll a_k$ and establishes (W1) for $\ll$.
  
For (W3),  let $a'\ll a$ and $b'\ll b$ with $a\in \CO_n(X)$ and $b\in \CO_m(X)$. Then simply observe that $(a', b')\ll_{m+n} (a, b)$ by definition, which implies that $a'+b'\ll a+b$. This establishes (W3). For (W4),  let $a, b, c\in \CK(X, \Gamma)$ such that $b=(U_1,\dots, U_n)\in \CO_n(X)$ and $c=(V_1,\dots, V_m)\in \CO_m(X)$ satisfying $a\ll b+c=(b, c)$. Then, one necessarily has $a=(O_1,\dots, O_{m+n})\in \CO_{m+n}(X)$ for some precompact open sets $O_i$ such that $\overline{O_i}\subset U_i$ for all $1\leq i\leq n$ and $\overline{O_i}\subset V_i$ for all $n+1\leq i\leq n+m$. Write $d_1=(O_1, \dots, O_n)$ and $d_2=(O_{n+1}, \dots, O_{n+m})$. For $1\leq i\leq n$, choose precompact open sets $M_i$ such that $\overline{O_i}\subset M_i\subset \overline{M_i}\subset U_i$. In addition, for $n+1\leq i\leq n+m$, choose precompact open sets $N_i$ such that $\overline{O_i}\subset N_i\subset \overline{N_i}\subset V_i$. Define $b'=(M_1, \dots, M_n)$ and $c'=(N_1, \dots, N_m)$. Our construction implies that $a=(d_1, d_2)\ll_{n+m} (b', c')$ and $b'\ll_n b$ as well as $c'\ll_m c$. This establishes (W4).

Then, Lemma \ref{lem: uniform}(ii) implies that $(\CK(X, \Gamma), \prec)$ satisfies (W1)-(W4).
\end{proof}

Now we define a relation $\preceq$ on $\CW(X, \Gamma)$ by $[a]\preceq [b]$ if $a\prec b$ in $\CK(X, \Gamma)$. This relation is well-defined. Indeed, suppose $a\prec b$ holds in $\CK(X, \Gamma)$ and $b\in \CO_n(X)$. Then there is an element $c\in \CO_n(X)$ such that $a\preccurlyeq c\ll_n b$. Now let $a_1\approx a$ and $b_1\approx b$ in $\CK(X, \Gamma)$ where $b_1\in \CO_m(X)$. Then first one has $a_1\preccurlyeq a\preccurlyeq c\ll_n b\preccurlyeq b_1.$
Then  Proposition \ref{prop: relation} implies that there is a $d\in \CO_m(X)$ with $d\ll_m b_1$ such that $c\preccurlyeq d\ll_m b_1$. Then Lemma \ref{lem:transitivity} implies that $a_1\preccurlyeq d\ll_m b_1$. This shows that $\preceq$ is well-defined.

\begin{cor}\label{cor: main1}
	The positively partially ordered monoid $\CW(X, \Gamma)$ is a \textnormal{W}-semigroup. 
\end{cor}
\begin{proof}
Let $[a], [b]\in \CW(X, \Gamma)$. Recall that, by definition, $[a]\leq [b]$ if and only if $a\preccurlyeq b$ in $\CK(X, \Gamma)$. In addition, $[a]\preceq [b]$ holds if and only if $a\prec b$ in $\CK(X, \Gamma)$. Then, using Theorem \ref{thm: K is W semigroup}, it is direct to see $\preceq $ is auxiliary for $\leq$ and (W1)-(W4) hold for $\preceq$ in $\CW(X, \Gamma)$.
\end{proof}

\section{Additional axioms}
A positively partially ordered monoid $S$ is called a Cu-semigroup if $(S, \leq)$ satisfies the axioms (O1)-(O4), introduced in \cite{CEI}, with respect to the way-below relation $\ll$. Roughly speaking, axiom (O1) means that every increasing sequence $\{a_n: n\in\N\}$ in $S$ has a supremum, whilst (O2) means that every $a\in S$ is the supremum of a $\ll$-increasing sequence. Axiom (O3) says that the addition and $\ll$ are compatible, and (O4) stands for compatibility of suprema and addition. See more detail in \cite{CEI}. As we have mentioned in the introduction, there is a functor $\gamma: \textnormal{W}\to \textnormal{Cu}$ so that for each W-semigroup $S$, the assigned Cu-semigroup $\gamma(S)$ serves as a completion of $S$ in the sense of \cite[Definition 3.1.7]{APT} by using \cite[Proposition 3.1.6]{APT}. Therefore, Corollary \ref{cor: main1} implies that $\gamma(\CW(X, \Gamma))$ is a Cu-semigroup.

Besides axioms (W1)-(W4), there are other axioms that a W-semigroup may satisfy. The following are also introduced in \cite{APT}, and are referred to as axioms (W5) and (W6). Let $(S, \leq, \prec)$ be a W-semigroup with an auxiliary relation $\prec$.
\begin{enumerate}
	\item[(W5)] For every $a', a, b', b, c, \tilde{c}$ that satisfy
	\[a+b\prec c,\ \  a'\prec a,\ \ b'\prec b,\ \ c\prec \tilde{c} \]
	there exist elements $x'$ and $x\in S$ such that 
	\[a'+x\prec \tilde{c},\ \ c\prec a+x', \ \ b'\prec x'\prec x.\]
	\item[(W6)] For every $a', a, b, c\in S$ that satisfy $a'\prec a\prec b+c$, there exist elements $e, f\in S$ such that $a'\prec e+f$, $e\prec a, b$ and $f\prec a, c$.
\end{enumerate}

We next show that $\CW(X, \Gamma)$  always satisfies (W6) with respect to the order $\preceq$.
\begin{prop}
	$\CW(X, \Gamma)$ satisfies \textnormal{(W6)}. 
\end{prop}
\begin{proof}
Let $a'\in \CO_n(X), a=(U_1, \dots, U_m)\in \CO_m(X), b=(B_1, \dots, B_k)\in  \CO_k(X)$ and $c=(B_{k+1}, \dots, B_{k+l})\in \CO_l(X)$ such that $[a']\preceq [a]\preceq [b]+[c]$ in $\CW(X, \Gamma)$. Then there is an element $d=(O_1, \dots, O_m)\in \CO_m(X)$ with each $O_i$ precompact such that $a'\preccurlyeq d\ll_m a$. In addition, there is an element $g=(C_1, \dots, C_{k+l})\in \CO_{k+l}(X)$ with each $C_j$ precompact such that $a\preccurlyeq g\ll_{k+l} (b, c)$. Then since $d\ll_m a$ and $a\preccurlyeq g$, and using normality of the space $X$, we may choose precompact open sets $\{V^i_j: 1\leq i\leq m, 1\leq j\leq J_i\}$, group elements $\{s^i_j\in \Gamma: 1\leq i\leq m, 1\leq j\leq J_i\}$ and integers $\{1\leq p^i_j\leq k+l: 1\leq i\leq m, 1\leq j\leq J_i\}$ such that $\overline{O_i}\subset \bigcup_{j=1}^{J_i}V^i_j\subset \bigcup_{j=1}^{J_i}\overline{V^i_j}\subset U_i$ for each $1\leq i\leq m$ and 
\[\bigsqcup_{i=1}^m\bigsqcup_{j=1}^{J_i}s^i_j\overline{V^i_j}\times \{p^i_j\}\subset \bigsqcup_{q=1}^{k+l}C_q\times \{q\}.\]
Now  define $e_i=\bigcup\{V^i_j: p^i_j\leq k, 1\leq j\leq J_i \}$ and $f_i=\bigcup\{V^i_j: p^i_j>k, 1\leq j\leq J_i \}$ for each $i\leq m$ and set $e=(e_1, \dots, e_m)$ and $f=(f_1, \dots, f_m)$, which are elements in $\CO_m(X)$.  Then one has $e_i\cup f_i=\bigcup_{j=1}^{J_i}V^i_j\subset \bigcup_{j=1}^{J_i}\overline{V^i_j}\subset U_i $ for each $i\leq m$, which implies that $e\ll_m a$ and $f\ll_m a$ and thus both $[e], [f]\preceq [a]$. In addition, since  $C_q$ are precompact and $\overline{C_q}\subset B_q$ for all $q\leq k+l$ and the set \[\bigsqcup\{s^i_j\overline{V^i_j}\times \{p^i_j\}: 1\leq i\leq m, 1\leq j\leq J_i, p^i_j\leq k\}\]
is a subset in $\bigsqcup_{q=1}^{k}C_q\times \{q\}$, one actually has $[e]\preceq [b]$. The same argument shows that $[f]\preceq [c]$.

Finally, by our construction, for each $i\leq m$ one has $\overline{O_i}\subset \bigcup_{j=1}^{J_i}V^i_j=e_i\cup f_i$. Write $t=(e_1\cup f_1, \dots, e_m\cup f_m)$ for simplicity and one has $a'\preccurlyeq d\ll_m t\preccurlyeq (e, f)$. This implies that $[a']\preceq [t]\leq [e]+[f]$. Since $\preceq$ is auxiliary, one actually has $[a']\preceq [e]+[f]$.
\end{proof}

Finally, under the additional assumption that $X$ is zero-dimensional (which means that the topology on $X$ has a basis consisting of compact open sets), we show that $\CW(X, \Gamma)$ satisfies (W5). In this case, one even may choose $[x']=[x]$ as compact elements in $\CW(X, \Gamma)$ with respect to $\preceq$.

\begin{prop}
If $X$ is zero-dimensional, then $\CW(X, \Gamma)$ satisfies \textnormal{(W5)}.
\end{prop}
\begin{proof}
Let $[a'], [a], [b'], [b], [c], [\tilde{c}]\in \CW(X, \Gamma)$ be such that $[a]+[b]\preceq [c]$, $[a']\preceq [a]$, $[b']\preceq [b]$ and $[c]\preceq [\tilde{c}]$. First write $a=(O_1, \dots, O_m)$ and $b=(O_{m+1}, \dots, O_{m+n})$. Then since $X$ is zero-dimensional, there are compact open sets $F_i\subset O_i$ for $i\leq m+n$ such that $a'\preccurlyeq (F_1, \dots, F_m)$ and $b'\preccurlyeq (F_{m+1}, \dots, F_{m+n})$. In addition,  there is an element $c'=(N_1, \dots, N_l)$ in which all $N_p$ are also compact open such that $c\preccurlyeq c'\ll_l \tilde{c}$. Because $(a, b)\preccurlyeq c$ and $X$ is zero-dimensional, for each $i\leq m+n$ and $F_i\subset O_i$, there are compact open sets $\{K^i_j: 1\leq j\leq J_i, 1\leq i\leq m+n\}$, group elements $\{s^i_j\in \Gamma: 1\leq j\leq J_i, 1\leq i\leq m+n\}$ and positive integers $\{k^i_j\leq l: 1\leq j\leq J_i, 1\leq i\leq m+n\}$ such that $\{K^i_j: 1\leq j\leq J_i\}$ is a disjoint family and $F_i\subset \bigsqcup_{j=1}^{J_i}K^i_j\subset O_i$ for each $1\leq i\leq m+n$ as well as
\[\bigsqcup_{i=1}^{m+n}\bigsqcup_{j=1}^{J_i}s^i_jK^i_j\times \{k^i_j\}\subset \bigsqcup_{p=1}^{l}N_p\times \{p\}. \tag{\ensuremath{\bigstar}}\label{eq1}\]
On the other hand, for each $p\leq l$, write $R_p=\bigsqcup\{s^i_jK^i_j: 1\leq i\leq m, 1\leq j\leq J_i, k^i_j=p\}$ for simplicity and define  $H_p=N_p\setminus R_p$, which is still a compact open set. Now define  $x=x'=(H_1, \dots, H_l)$.  Note that $[x']\preceq [x]$ since all $H_p$ are compact open.

Now, let $1\leq p\leq l$. For any compact set $T_p\subset N_p$, observe that $T_p\subset (T_p\cap R_p)\cup H_p$. For any $1\leq i\leq m$ and $1\leq j\leq J_i$ with $k^i_j=p$, define $t^p_{i, j}=(s^i_j)^{-1}\in \Gamma$, $L^p_{i, j}=s^i_jK^i_j$ and $d^p_{i, j}=i$. Then  one actually has 
\[\bigsqcup_{p=1}^l\bigsqcup_{i=1}^{m}\bigsqcup_{j=1}^{J_i}t^p_{i, j}L^p_{i, j}\times \{d^p_{i, j}\}=\bigsqcup_{i=1}^{m}(\bigsqcup_{j=1}^{J_i}K^i_j)\times \{i\}\subset \bigsqcup_{i=1}^{m}O_i\times \{i\}\]
while $T_p\cap R_p\subset R_p= \bigsqcup\{L^p_{i, j}: 1\leq i\leq m, 1\leq j\leq J_i, k^i_j=p\}$. This implies that $c\preccurlyeq (\bigsqcup_{j=1}^{J_1}K^1_j, \dots,\bigsqcup_{j=1}^{J_m}K^m_j, H_1, \dots, H_l)\ll_{m+l} (a, x')$, which means that $[c]\preceq [a]+[x']$. On the other hand,  observe that 
\[(a', x)\preccurlyeq (F_1, \dots, F_m, H_1, \dots, H_l)\preccurlyeq (R_1, \dots, R_l, H_1, \dots, H_l)\preccurlyeq c'\ll_l \tilde{c}\] because $R_p\sqcup H_p=N_p$ for all $p\leq l$. This implies that $[a']+[x]\preceq [\tilde{c}]$.

Finally, for each $p\leq l$, define  $Q_p=\bigsqcup\{s^i_jK^i_j: m+1\leq i\leq m+n, 1\leq j\leq J_i, k^i_j=p\}$, which is a compact open set. Observe that $Q_p\subset H_p$ by (\ref{eq1}), which implies that 
\[b'\preccurlyeq (F_{m+1}, \dots, F_{m+n})\preccurlyeq (Q_1, \dots, Q_l)\ll_l(H_1, \dots, H_l)=x'.\]
This entails that $[b']\preceq [x']\preceq [x]$.
\end{proof}

To summarize, by combining the results in \cite[Theorem 4.4]{APT} and Corollary \ref{cor: main1}, we have the following.
\begin{thm}
The semigroup $\CW(X, \Gamma)$ satisfies \textnormal{(W1)-(W4)} and \textnormal{(W6)} and thus its completion $\gamma(\CW(X, \Gamma))$ is a \textnormal{Cu}-semigroup and satisfies \textnormal{(O6)}. If $X$ is a zero-dimensional space, the semigroup $\CW(X, \Gamma)$ additionally satisfies \textnormal{(W5)} and thus its completion $\gamma(\CW(X, \Gamma))$ additionally satisfies \textnormal{(O5)}.
\end{thm}

\section{Acknowledgment}
The author should like to thank Hanfeng Li for his valuable comments and corrections. In addition, the author should like to thank the referee for the careful reading and very helpful suggestions, which lead to the current more conceptual approach in establishing Corollary \ref{cor: main1}. These have improved the paper a lot.

\end{document}